\documentclass[final,leqno]{siamltex}

\newtheorem{rem}[theorem]{Remark}
\newtheorem{exs}[theorem]{Examples}

\usepackage{lscape}
\usepackage{amsmath,amssymb}
\usepackage{graphicx}
\usepackage{exscale,cmmib57}

\title{Extremal and optimal properties of B-bases Collocation Matrices \thanks{Received; accepted for publication; published electronically. This work was partially supported by the Spanish Research grant
MTM2015-65433-P (MINECO/FEDER), Gobierno de
Arag\'on, and Fondo Social Europeo.
\URL simax/31-3/73797.html}}

\author{Jorge Delgado\thanks{Departamento de Matem\'atica Aplicada, Universidad de Zaragoza, Escuela Universitaria Polit\'ecnica de Teruel, E-44071 Teruel, Spain (jorgedel@unizar.es).}
        \and J. M. Pe\~na\thanks{Departamento de Matem\'atica Aplicada, Universidad de Zaragoza, E-50009 Zaragoza, Spain (jmpena@unizar.es).}}

\begin{document}
\maketitle

\begin{abstract}
Totally positive matrices are related with the shape preserving representations of a space of functions. The normalized B-basis of the space has optimal shape preserving properties. B-splines and rational Bernstein bases are examples of normalized B-bases. Some results on the optimal conditioning and on extremal properties of the minimal eigenvalue and singular value of the collocation matrices of  normalized B-bases are proved. Numerical examples confirm the theoretical results and answer related questions.
\end{abstract}




\pagestyle{myheadings} \thispagestyle{plain} \markboth{JORGE
DELGADO AND J. M. PE\~{N}A}{OPTIMAL PROPERTIES OF B-BASES COLLOCATION MATRICES}

\section{Introduction}\label{sect1}

 Totally positive matrices, which are also called totally nonnegative in
the literature,  play an important role in many fields, such as approximation
theory, computer aided geometric design (CAGD), mechanics, differential or
integral equations, statistics, combinatorics, economics and biology (see
\cite{a},  \cite{fj}, \cite{g-m}, \cite{k} or \cite{pi}). A matrix is {\it totally positive} (TP) if all its minors are
nonnegative. Relevant properties of TP matrices about algebraic computations with high relative accuracy have been found recently (cf. \cite{dk,ko}). In fact, for some classes of TP matrices adequately parameterized, one can compute their eigenvalues, singular values, inverses or the solutions of some linear systems with high relative accuracy independently of their conditioning (see \cite{dk}, \cite{m-m} and \cite{dp3}). This holds for many
popular matrices, such as positive Vandermonde matrices or Hilbert
matrices, which are TP. An important source of examples of TP matrices comes from
the collocation matrices of systems of functions. Let $\cal U$ be
a vector space of real functions defined on a real interval $I$
and $(u_0(t),\ldots,u_n(t))$ ($t\in I$) be a basis of $\cal U$.
The {\it collocation matrix} of $(u_0(t),\ldots,u_n(t))$ at $t_0
<\cdots < t_m$ in $I$ is given by
\begin{equation}
    M \left( {u_0, \ldots, u_n \atop t_0, \ldots, t_m} \right):= (u_j(t_i))_{i=0, \ldots, m;j=0,\ldots, n}.
\end{equation}
The collocation matrices of a given basis are the coefficient matrices of the linear systems
associated with Lagrange interpolation problems in that
basis.

A system of functions is TP when all its collocation matrices
(1.1) are TP. 
In CAGD, the functions $u_0,\ldots ,u_n$ also satisfy
that  $\sum _{i=0}^nu_i(t)=1$ $\forall \, t\in I$ (i.e., the
system $(u_0,\ldots,u_n)$ is  {\it normalized}), and a normalized
TP system is denoted by NTP. It is known that shape preserving
representations are associated with NTP bases (see \cite{p} or
\cite{c-p}). Clearly,  the collocation matrices of NTP bases are stochastic TP matrices. By Theorem 4.2 (ii) of \cite{c-p2} (see also \cite{TPandOB,p}), given a space with an NTP basis, there exists a unique NTP basis of the space with optimal shape preserving properties, which is called the {\it normalized B-basis} of the space.
An important normalized B-basis is the Bernstein basis $(b_0^n,\ldots,b_n^n)$ of the space $\mathcal{P}_n([0,1])$ of
polynomials of degree less than or equal to $n$ on $[0,1]$, given by
\begin{equation}\label{eq:Ber}
 b_i^n(t)={n \choose i}t^i(1-t)^{n-i},\quad i=0,1,\ldots,n
\end{equation}
(see \cite{c-p}, \cite{c-p2}).
Other examples of normalized B-bases are presented at the end of Section 3 and include the important examples of B-splines and rational Bernstein bases.

In this paper, we prove that the minimal eigenvalue (and singular value) of a collocation matrix of an NTP basis is always bounded above by the minimal eigenvalue (and singular value, respectively) of the corresponding collocation matrix of the normalized B-basis of the space. The information on the minimal eigenvalue and singular value has important potential applications. For instance, here we extend the optimal conditioning for the $\infty$-norm of the Bernstein basis proved in \cite{dp2} to any normalized B-basis. On the other hand, similar results for the maximal singular value of the corresponding collocation matrices do not hold, as shown in Section 4.

The paper is organized as follows. Section 2 presents basic concepts and notations, as well as some auxiliary results for TP matrices. In particular, it recalls the characterization of stochastic TP matrices as a product of matrices associated with elementary corner cuttings. In Section 3, we prove that multiplying a nonsingular TP matrix by a matrix associated with an elementary corner cutting decreases the minimal eigenvalue and singular value and increases the $\infty$-norm condition number. This result is a key tool to prove the mentioned result on the extremal and optimal properties of the collocation matrices of a normalized B-basis. In Section 4, we include numerical examples confirming our theoretical results and counterexamples answering other related 
questions.

\section{Basic notations and auxiliary results}

By Theorem 2.6  of \cite{p} (or by Theorem 4.5 of \cite{kl}) we have 
the following characterization of a nonsingular stochastic TP matrix.
\begin{theorem}\label{thm:fact.TPyEst}
	A nonsingular $n\times n$ matrix $A$ is stochastic and TP
	if and only if it can be factorized in the form
	\[
		A=F_{n-1}F_{n-2}\cdots F_1G_1\cdots G_{n-2}G_{n-1},
	\]
	with
    \[
    	F_i = \left(\begin{matrix} 
    	1 \\ 
    	0 & 1 \\
		 &\ddots &\ddots \\
		 &  & 0 & 1 \\
		 & &    & \alpha_{i+1,1} & 1- \alpha_{i+1,1}\\
		 & & & & \ddots & \ddots \\
		   &   &  &  & &\alpha_{n,n-i} & 1- \alpha_{n,n-i}
		\end{matrix}\right)
	\] 
	and
	\[
	G_i = \left(\begin{matrix} 
	1 &0\\
	 &\ddots &\ddots \\
	   &  & 1 &0 \\
	  &    &  & 1- \alpha_{1,i+1}&\alpha_{1,i+1}\\
	 & & & &\ddots & \ddots \\
	    & &  &  & &  1-\alpha_{n-i,n} &\alpha_{n-i,n}\\
	& &  &  & & & 1
	\end{matrix}\right),
	\] 
	where, $\forall\, (i,j),$ $0\le \alpha_{i,j}<1$.
\end{theorem}

The following remark provides a new factorization in terms of elementary bidiagonal matrices.

\begin{rem}\label{rem:fact.TPyEst}
	If we denote by $U_i(\lambda)$ the bidiagonal, nonsingular
	and upper triangular matrix with at most one nonzero
	off-diagonal element in the entry $(i-1,i)$
	\begin{equation}\label{eq:el.u}
		U_i(\lambda)=\left(\begin{matrix}
		1&0&&&&&\\
		&1&0&&&&\\ 
		&&\ddots &\ddots &&&\\ 
		&&&1-\lambda&\lambda&&\\
		&&&&\ddots &\ddots &\\ 
		&&&&&1&0\\ 
		&&&&&&1
		\end{matrix}\right), \quad  0\le \lambda<1
	\end{equation}
	and by $L_i(\lambda)$ the bidiagonal, nonsingular
	and lower triangular matrix with at most one nonzero
	off-diagonal element in the entry $(i,i-1)$
	\begin{equation}\label{eq:el.l}
		L_i(\lambda)=\left(\begin{matrix}
			1&&&&&&\\
			0&1&&&&&\\ 
			&\ddots &\ddots & &&&\\ 
			&&\lambda&1-\lambda&&&\\
			&&&\ddots &\ddots & &\\ 
			&&&&&1&\\ 
			&&&&&0&1
			\end{matrix}\right), \quad 0\le \lambda<1, 
	\end{equation}
	then we can write
	\[
		F_i=L_{i+1}(\alpha_{i+1,1} )\cdots L_n(\alpha_{n,n-i}) \quad\text{and}\quad 	G_i=U_n(\alpha_{n-i,n})\cdots U_{i+1}(\alpha_{1,i+1}).
    \]
\end{rem}

In Section 2 of \cite{p}, it is shown that the elementary matrices (\ref{eq:el.u}) and (\ref{eq:el.l}) have a geometric interpretation as elementary corner cutting transformations.

Now let us recall some notations, concepts and results of Linear Algebra that will be used later in order to
get a paper as self-contained as possible. Given two square matrices $A=(a_{ij})_{1\le i,j\le n}$ and $B=(b_{ij})_{1\le i,j\le n}$, we denote $A\le B$ if $a_{ij}\le b_{ij}$ for all $i,j$. We say that $A$ is {\it nonnegative} if $a_{ij}\ge 0$ for all $i,j$.  If $C=(c_{ij})_{1\le
i,j\le n}$ is a complex matrix and $A=(a_{ij})_{1\le i,j\le n}$ is a
nonnegative matrix such that $|c_{ij}|\le a_{ij}$ for all $i,j$,
then $A$ is said to {\it dominate} $C$, and so $|C|:=(|c_{ij}|)_{1\le i,j\le n}\leq A$.
The following result is due to Wienlandt (see Corollary 2.1 of Chapter
II of \cite{M}):

\begin{theorem}\label{W} Let $M$ be a nonnegative matrix  with maximal eigenvalue $r$, and let
$C$ be a complex matrix dominated by $M$. Then $r=\rho (M)\ge \rho (C)$.
\end{theorem}

The following result collects two properties of TP matrices which will be used in the proofs of the main results. The first part corresponds to Corollary 6.6 of \cite{a} and the second part to
Theorem 3.3 of \cite{a}.

\begin{theorem}\label{TP} 
Let $A$ be a nonsingular TP $n\times n$ matrix. Then:
\begin{itemize}
 \item[(i)] All the eigenvalues of $A$ are positive.
 \item[(ii)] Given the $n\times n$ diagonal matrix 
 \begin{equation}\label{eq:j}
 J:=\hbox{diag}
(1,-1,1,\ldots ,(-1)^{n-1}),
\end{equation}
 the matrix $JA^{-1}J$ is TP.
 \end{itemize}
\end{theorem}

Given a nonsingular matrix $A$, for $p=1,2,\infty$ we shall use the condition numbers $\kappa _p(A):=\Vert A\Vert _p\, \Vert A^{-1}\Vert _p$.

\section{Main results}\label{sec:main}


The following theorem shows that the elementary matrices corresponding to elementary corner cuttings decrease the minimal singular value and the minimal eigenvalue and increase some condition numbers when they multiply a TP matrix to its right or when their transposes multiply a TP matrix to its left.

\begin{theorem}\label{lem:aux}
Let $M$ be a nonsingular TP matrix,  $A:=ME$
and $C:=E^TM$ with $E=U_i(\lambda)$ or $E=L_i(\lambda)$ an elementary matrix given by 
\eqref{eq:el.u} and \eqref{eq:el.l}, respectively, for $0\le\lambda <1$.
Then the following properties hold:
\begin{itemize}
 \item[(i)] $|A^{-1}|$ and $|C^{-1}|$ dominate $M^{-1}$.
 \item[(ii)] The minimal eigenvalue of $A$ and $C$ are bounded above by the minimal
eigenvalue of $M$.
 \item[(iii)] The minimal singular value of $A$ and $C$ are bounded above by the minimal
singular value of $M$.
 \item[(iv)] $\kappa _\infty (M)\le \kappa _\infty (A)$ and $\kappa _1(M)\le \kappa _1 (C)$
  \end{itemize}
\end{theorem}
\begin{proof}
Since $E$ is obviously TP and $M$ is also TP, we deduce from Theorem 3.1 of \cite{a} that the products $A=ME$ and $C=E^TM$ are also TP, and they also inherit the nonsingularity of $M$ and $E$. 
If $J$ is the diagonal matrix given by (\ref{eq:j}), since $A$, $C$ and $M$ are TP nonsingular, by Theorem \ref{TP} (ii), $JA^{-1}J$, $JC^{-1}J$ and $JM^{-1}J$ are TP and so, in particular, nonnegative and 
\begin{equation}\label{eq:vabs}
JA^{-1}J=|A^{-1}|, \quad JC^{-1}J=|C^{-1}|. 
\end{equation}
Besides, $JA^{-1}J$, $JC^{-1}J$ and $JM^{-1}J$ are similar to $A^{-1}$, $C^{-1}$ and
 $M^{-1}$, respectively.

(i) Taking into account 
that $J=J^{-1}$ we derive  
\[
JA^{-1}J=J(ME)^{-1}J=(JE^{-1}J)(JM^{-1}J).
\]
 So,
in order to prove that  $|A^{-1}|$ dominates $M^{-1}$, it is sufficient by (\ref{eq:vabs}) to see that
\begin{equation}\label{eq:aux.lem}
JM^{-1}J\leq JA^{-1}J = (JE^{-1}J)(JM^{-1}J).
\end{equation}
We can observe that the matrix $JE^{-1}J$ is also nonnegative. In addition,
$JE^{-1}J$ has one of the two following forms:
\begin{equation}\label{eq:inv.Jel}
\begin{pmatrix}
1&0&&&&&\\
&1&0&&&&\\ 
&&\ddots &\ddots &&&\\ 
&&&\frac{1}{1-\lambda }&\frac{\lambda}{1-\lambda}&&\\
&&&&\ddots &\ddots &\\ 
&&&&&1&0\\ 
&&&&&&1
\end{pmatrix} \ \  \text{or} \ \
\begin{pmatrix}
1&&&&&&\\
0&1&&&&&\\ &\ddots &\ddots & &&&\\ &&\frac{\lambda}{1-\lambda}&\frac{1}{1-\lambda }&&&\\
&&&\ddots &\ddots & &\\ &&&&&1&\\ &&&&&0&1
\end{pmatrix}, 
\end{equation}
with $0\leq\lambda<1$.
Taking into account the previous formula, that $JM^{-1}J$ is nonnegative
and that $1/(1-\lambda)\ge 1$, it can be deduced
that $(JE^{-1}J)(JM^{-1}J)\ge JM^{-1}J$ and formula \eqref{eq:aux.lem} holds, and so $|A^{-1}|$ dominates $M^{-1}$. Since $C^T=M^TE$, we can deduce that $|(C^T)^{-1}|=|(C^{-1})^T|$ dominates $(M^T)^{-1}=(M^{-1})^T$ and so $|C^{-1}|$ dominates $M^{-1}$, and (i) holds. 

(ii) By Theorem \ref{TP} (i), the eigenvalues of $A$ are positive. By (\ref{eq:vabs}), (i) and Theorem \ref{W}, we derive
\begin{equation}\label{eq:equiv}
	\rho(JA^{-1}J)\ge \rho(JM^{-1}J)
\end{equation}
and, since $JA^{-1}J$ and $JM^{-1}J$ are similar to $A^{-1}$ and $M^{-1}$ (respectively), the minimal eigenvalue of $A$ is bounded above by the minimal
eigenvalue of $M$. Using again that $C^T=M^TE$ and that the eigenvalues
do not change when transposing a matrix, we also conclude that the minimal eigenvalue of $C$ is bounded above by the minimal
eigenvalue of $M$, and (ii) holds.

(iii) The minimal singular values of $M$ and $A=ME$ are the minimal eigenvalues of $M^TM$ and $E^TM^TME$, respectively. By Theorem 3.1 of \cite{a}, the product $M^TM$ is TP. Then, by (ii), the minimal eigenvalue of $M^TM$ is greater than or equal to the minimal eigenvalue of $M^TME$. Applying (i) again, the minimal eigenvalue of $M^TME$ is greater than or equal to the minimal eigenvalue of $E^TM^TME$. In conclusion, the minimal singular value of $A=ME$ is bounded above by the minimal
singular value of $M$. Taking into account that $C^T=M^TE$ and that the singular values do not change when transposing a matrix,  we also have that the minimal singular value of $C$ is bounded above by the minimal
singular value of $M$, and (iii) holds.

(iv) From (i), we derive $\Vert M^{-1}\Vert _\infty \le \Vert A^{-1}\Vert _\infty $. Since $A$ and $M$ are TP, they are nonnegative. Since $E$ is stochastic, if we denote $e:=(1,\ldots,1)^T$, then we have 
$\Vert A\Vert _\infty =\Vert Ae\Vert _\infty=\Vert MEe\Vert _\infty=\Vert Me\Vert _\infty=\Vert M\Vert _\infty$. Therefore $\kappa _\infty (M)\le \kappa _\infty (A)$. Finally, we deduce that $\kappa _1(C)=\kappa _\infty (C^T)=\kappa _\infty (M^TE)\ge \kappa _\infty (M^T)=\kappa _1(M)$, and the result follows.
\end{proof}

The following corollary shows that any nonsingular stochastic TP matrix produces the same effects as those described in Theorem \ref{lem:aux} for the elementary matrices corresponding to elementary corner cuttings when multiplying TP matrices.

\begin{corollary}\label{cor:aux}
Let $M$ be a nonsingular TP matrix, $K$ a nonsingular stochastic TP matrix, $A:=MK$ and $C:=K^TM$. Then the following properties hold:
\begin{itemize}
 \item[(i)] $|A^{-1}|$ and $|C^{-1}|$ dominate $M^{-1}$.
 \item[(ii)] The minimal eigenvalue of $A$ and $C$ are bounded above by the minimal
eigenvalue of $M$.
 \item[(iii)] The minimal singular value of $A$ and $C$ are bounded above by the minimal
singular value of $M$.
 \item[(iv)] $\kappa _\infty (M)\le \kappa _\infty (A)$ and $\kappa _1(M)\le \kappa _1 (C)$
  \end{itemize}
\end{corollary}
\begin{proof}
By Theorem \ref{thm:fact.TPyEst} and Remark \ref{rem:fact.TPyEst} 
we deduce that $K=\prod_{i=1}^r E_i$,
where $r$ is a positive integer and each $E_i$ is equal to $U_j(\lambda_i)$
or $L_j(\lambda_i)$ given by \eqref{eq:el.u} and \eqref{eq:el.l}, respectively,
for $0\leq\lambda_i<1$. Therefore, we get that
\[
	A=M\left(\prod_{i=1}^r E_i\right),
\]
with $E_i=U_j(\lambda_i)$ or $L_j(\lambda_i)$ for $0\leq\lambda_i<1$ and $i\in\{1,\ldots,r\}$.
So, applying in an iterative way Theorem \ref{lem:aux} to the previous formula,
the result follows for $A$.

Analogously, since $C=(E_r^T\cdots E_1^T)M$ with each $E_i$ 
a matrix of the form \eqref{eq:el.u} or \eqref{eq:el.l},
we can apply Theorem \ref{lem:aux} in an iterative way to prove the
result for $C$.
\end{proof}

The next corollary applies previous results to deduce some extremal and optimal properties of the collocation matrices of the normalized B-basis of a space.

\begin{corollary}\label{cor:eig}
Let $u=(u_0,\ldots,u_n)$ be an NTP basis on $[a,b]$ of a space of functions $\mathcal{U}$ and 
let $v=(b_0,\ldots,b_n)$ the normalized B-basis of $\mathcal{U}$. If we consider an increasing sequence of 
nodes $\mathbf{t}=(t_i)_{i=0}^n$ on $[a,b]$, let us denote by $A$ to the collocation matrix of $u$ at 
$\mathbf{t}$ and by $M$ to the collocation matrix of $v$ at $\mathbf{t}$.  
Then the minimal 
eigenvalue and singular value of $M$ are greater than or equal to the minimal eigenvalue and singular value of $A$, respectively.
Moreover, if $A$ and $M$ are nonsingular, then 
$\kappa_1(M^T)=\kappa _\infty (M)\le \kappa _\infty (A)=\kappa_1(A^T)$.
\end{corollary}
\begin{proof}
Since $v$ is the normalized B-basis of $\mathcal{U}$ and $u$ an NTP basis, by Theorem 4.2 (ii) of \cite{c-p2}, we have that there exists a nonsingular TP
stochastic matrix $K$ such that
\[
	(u_0,\ldots,u_n)=(b_0,\ldots,b_n)K.
\]
Taking collocation matrices in the previous expression at $\mathbf{t}$ we have that 
\begin{equation}\label{eq:rel.eig}
	A=MK.
\end{equation}
Since the bases $u$ and $v$ are NTP, $A$ and $M$ are stochastic and TP. 
If $A$ (or equivalently $M$) is singular, then the minimal eigenvalue and singular value of both matrices are equal to 0.  Otherwise, 
 the result follows from (\ref{eq:rel.eig}) and from (ii), (iii) and (iv) of Corollary \ref{cor:aux}. 
\end{proof}

We now give a list of examples of important normalized B-bases. By the previous result, their collocation matrices satisfy the mentioned extremal and optimal properties.

\begin{exs}\label{exs:NBb}
	\begin{itemize}
		
		\item[(a)] The space of polynomials of degree at most $n$ on a compact interval $[a,b]$,
			$\mathcal{P}_n([a,b])$, has the normalized B-basis given by $(b_0^n\,\ldots,b_n^n)$ 
			with 
			\[
				b_i^n(t;a,b)={n\choose i}\frac{(b-t)^{n-i}(t-a)^i}{(b-a)^n},\quad i=0,1\ldots,n
			\]
	(see \cite{c-p2,f-r} and Section 4 of \cite{c-p}).
		
		\item[(b)] Let us consider a sequence $(w_i)_{0\leq i\leq n}$ of positive weights. Then
		the system of functions $(r_0^n,\ldots,r_n^n)$ defined on the compact interval $[a,b]$ by
		\[
			r_i^n(t)=\frac{w_ib_i^n(t;a,b)}{\sum_{j=0}^n w_j b_j^n(t;a,b)},\quad i=0,1,\ldots,n,
		\]
		is the normalized B-basis of the corresponding spanned space of functions (see Example 4.14 of \cite{p}), and is called the rational Bernstein basis of its space. Observe that, if all weights $w_i=1$ for all $i$, then $(r_0^n,\ldots,r_n^n)=(b_0^n\,\ldots,b_n^n)$ is the Bernstein basis on $[a,b]$.
		
		\item[(c)] The space of even trigonometric functions given by
		\[
			\mathcal{C}_n=span\{1,\cos\,t,\cos\,2t,\ldots,\cos\,nt\}
		\]
		on the compact interval $[0,\pi]$ has the normalized B-basis 
		$(u_0^n,\ldots,u_n^n)$ given by 
		\[
			u_i^n(t)={n\choose i}\cos^{2(n-i)}(t/2)\sin^{2i}(t/2),\quad i=0,1,\ldots,n
		\]
(see \cite{trig}).
		
		\item[(d)] The space of trigonometric polynomials
		\[
			\mathcal{T}_n=\{1,\cos t,\sin t,\cos 2t,\sin 2t,\ldots,\cos nt,\sin nt\}
		\]
		on $I=[-A,A]$ with $A<\frac{\pi}{2}$ has the normalized B-basis $(v_0,\ldots,v_m)$, $m=2n$,
		defined by
		\[
			v_i(t)=d_i\left(\frac{\sin\left(\frac{A+t}{2}\right)}{\sin A}\right)^i
			\left(\frac{\sin\left(\frac{A-t}{2}\right)}{\sin A}\right)^{m-i},\quad i=0,1,\ldots,m
		\]
		with
		\[
			d_i=\sum_{k=0}^{[i/2]}{m/2 \choose i-k}{i-k\choose k}(2\cos A)^{i-2k},\quad i=0,1,\ldots,m
		\]
		(see Section 3 of \cite{sr}).
		
		\item[(e)]  A very important example is the case of B-spline bases (see \cite{s}) and NURBS. Let us consider
		a sequence of positive weights $(w_i)_{0\leq i\leq n}$ and a knots vector $(t_0,\ldots,t_{n+d})$
		with $t_i\leq t_{i+1}$ for all $i=0,1,\ldots,n+d-1$. Then the B-spline basis 
		$(N_{0,d},N_{1,d},\ldots,N_{n,d})$ defined
		over the previous knots vector by
		\begin{align*}
		N_{i,0}(t)&=\left\{
			\begin{array}{ll}
				1,& \text{if } t_i\leq t<t_{i+1}, \\
				0,& \text{otherwise,}
			\end{array}
		\right. \\
		N_{i,k}(t)&=\frac{t-t_i}{t_{i+k}-t_i}N_{i,k-1}(t)+\frac{t_{i+k+1}-t}{t_{i+k+1}-t_{i+1}}N_{i+1,k-1}(t),\quad k=1,\ldots,d,
		\end{align*}
		is the normalized B-basis of the corresponding splines space (see \cite{c-p2}). The basis
		$(r_0,\ldots,r_n)$ defined by
		\[
			r_i(t)=\frac{w_iN_{i,d}(t)}{\sum_{j=0}^n  w_jN_{j,d}(t)},\quad i=0,1\ldots,n,
		\]
		is the normalized B-basis of the corresponding NURBS space (see Section 4 of \cite{c-p2}).
	\end{itemize}
\end{exs}

\section{Numerical experiments and further questions}

For the construction of the numerical examples we shall consider three different
TP bases $u=(u_0^n,\ldots,u_n^n)$ of $\mathcal{P}_n([0,1])$. For each of the
bases, given a sequence of positive weights $(w_i)_{i=0}^n$, we can construct a rational NTP basis $(r_0^n,\ldots,r_n^n)$ defined by
\begin{equation}\label{eq:}
	r_i^n(t)=\frac{w_iu_i^n(t)}{\sum_{j=0}^n w_ju_j^n(t)},\quad i=0,1,\ldots,n.
\end{equation}
In fact, it is straightforward to check that, if $u$ is TP, then 
$(r_0^n,\ldots,r_n^n)$ is NTP. 
In the case that $u=(b_0^n,\ldots,b_n^n)$ is the 
normalized B-basis of the space $\mathcal{P}_n([0,1])$  given in  Examples \ref{exs:NBb} (a) for $a=0$ and $b=1$ (see \eqref{eq:Ber}),
then it is well known that the corresponding rational Bernstein basis $r_B=(r_0^n,\ldots,r_n^n)$ is the 
normalized B-basis of its spanned space $\left\langle r_B\right\rangle$ (see Examples \ref{exs:NBb} (b)).

Now, let us consider
the Said-Ball basis $s=(s_0^n,\ldots,s_n^n)$ 
(for more details see \cite{SB} and the references therein) given by 
\begin{align*}
	s_i^n(t)&={\lfloor n/2\rfloor + i\choose i}t^i(1-t)^{\lfloor n/2\rfloor +1},\quad 0\leq i\leq \lfloor (n-1)/2\rfloor, \\
	s_i^n(t)&={\lfloor n/2\rfloor + n-i\choose n-i}t^{\lfloor n/2\rfloor +1}(1-t)^{n-i},\quad \lfloor n/2\rfloor+1\leq i\leq n,
\end{align*}
and, if $n$ is even
\[
	s_{n/2}^n(t)={n\choose n/2}t^{n/2}(1-t)^{n/2},
\]
where $\lfloor m\rfloor$ is the greatest integer less than or equal to $m$. In \cite{SB} 
it was proved that the Said-Ball basis is NTP.
In the case that $u=(s_0^n,\ldots,s_n^n)$, the corresponding NTP basis $r_{SB}=(r_0^n,\ldots,r_n^n)$, constructed as in \eqref{eq:},
will be called rational Said-Ball basis.

Finally, let us consider the DP basis $c=(c_0^n,\ldots,c_n^n)$ of $\mathcal{P}_n([0,1])$ given by (see \cite{d-p})
\begin{align*}
	c_0^n(t)&=(1-t)^n, \\
	c_i^n(t)&=t(1-t)^{n-i},\quad 1\leq i\leq \lfloor n/2\rfloor -1, \\
	c_i^n(t)&=t^i(1-t),\quad \lfloor (n+1)/2\rfloor+1 \leq i\leq n-1, \\
	c_n^n(t)&=t^n,
\end{align*}
and, if $n$ is even
\[
	c_{\frac{n}{2}}^n(t)=1-t^{\frac{n}{2}+1}-(1-t)^{\frac{n}{2}+1},
\]
and, if $n$ is odd,
\begin{align*}
	c_{\frac{n-1}{2}}^n(t)&=t(1-t)^{\frac{n+1}{2}}+\frac{1}{2}\left[1-t^{\frac{n+1}{2}+1}-(1-t)^{\frac{n+1}{2}+1}\right], \\
	 c_{\frac{n+1}{2}}^n(t)&=\frac{1}{2}\left[1-t^{\frac{n+1}{2}+1}-(1-t)^{\frac{n+1}{2}+1}\right]+t^{\frac{n+1}{2}}(1-t).
\end{align*}
In \cite{d-p} it was also proved that the DP basis is also NTP.
In the case that $u=(c_0^n,\ldots,c_n^n)$, the corresponding basis $r_{DP}=(r_0^n,\ldots,r_n^n)$, constructed as in \eqref{eq:},
will be called rational DP basis.

As commented above, the rational Said-Ball and DP bases are also NTP.

If we consider a sequence of positive weights $(w_i^n)_{i=0}^n$ and taking into account that
$\sum_{j=0}^n w_j^n b_j^n(t)\in\mathcal{P}_n([0,1])$ and that $s$ and $c$ are bases of $\mathcal{P}_n([0,1])$,  
then there exist two sequences of weights $(\overline{w}_i^n)_{i=0}^n$ and $(\tilde{w}_i^n)_{i=0}^n$ satisfying
\begin{equation}\label{eq:rel.RB}
	\sum_{j=0}^n w_j^n b_j^n(t)=\sum_{j=0}^n \overline{w}_j^n s_j^n(t)=\sum_{j=0}^n \tilde{w}_j^n c_j^n(t),\quad t\in [0,1].
\end{equation}
If $\overline{w}_i^n,\tilde{w}_i^n>0$ for all $i=0,\ldots,n$, then
the rational Said-Ball basis $r_{SB}$ formed with the weights $(\overline{w}_i^n)_{i=0}^n$ and the rational
DP basis $r_{DP}$ formed with the weights $(\tilde{w}_i^n)_{i=0}^n$ are both NTP bases of the space
of rational functions $\langle r_B\rangle$, where $r_B$ is the rational Bernstein basis formed
with the weights $(w_i^n)_{i=0}^n$.
So, sequences of positive weights $(w_i^n)_{i=0}^n$ have been randomly generated for
each $n$ in $\{3,\ldots,8\}$, where each $w_i^n$ is an integer in the interval $[1,1000]$,
until we have obtained a sequence such that there exists 
positive sequences $(\overline{w}_i^n)_{i=0}^n$ and $(\tilde{w}_i^n)_{i=0}^n$ 
satisfying \eqref{eq:rel.RB}.
Then we have the normalized B-basis $r_B$, and the NTP bases $r_{SB}$
and $r_{DP}$ of $\langle r_B\rangle$. 

Let $(t_i)_{i=1}^{n+1}$ be the sequence of points given by $t_i=i/(n+2)$ for $i=1,\ldots,n+1$. Then we have 
considered the following collocation matrices:
\begin{align*}
	M^n&=\left(\frac{w_j^n b_j^n(t_i)}{\sum_{k=0}^n w_k^n b_k^n(t_i)}\right)_{1\leq i\leq n+1}^{0\leq j \leq n},\\
	B_1^n&=\left(\frac{\overline{w}_j^n s_j^n(t_i)}{\sum_{k=0}^n \overline{w}_k^n s_k^n(t_i)}\right)_{1\leq i\leq n+1}^{0\leq j \leq n}\quad\text{and}\quad 
	B_2^n=\left(\frac{\tilde{w}_j^n c_j^n(t_i)}{\sum_{k=0}^n \tilde{w}_k^n c_k^n(t_i)}\right)_{1\leq i\leq n+1}^{0\leq j \leq n},
\end{align*}
for $n=3,\ldots,8$.
We have computed the eigenvalues and the singular values of $M^n$, $B_1^n$ and $B_2^n$ 
for $n=3,\ldots,8$ with Mathematica
using a precision of $100$ digits. We can see the corresponding minimal eigenvalues and singular values in Table \ref{tab:min}.
It can be observed that the minimal eigenvalue, resp. singular value, of $M_n$ is higher than the minimal eigenvalue,
resp. singular value, of $B_1^n$ and $B_2^n$ as Corollary \ref{cor:eig} has proved.

\begin{table}[h!]
\centering\begin{tabular}{|r|cc|cc|cc|}
\hline
$n$ & \multicolumn{2}{|c|}{$M^n$} & \multicolumn{2}{|c|}{$B_1^n$} & \multicolumn{2}{|c|}{$B_2^n$} \\
       & $\lambda_{min}$ & $\sigma_{min}$ & $\lambda_{min}$ & $\sigma_{min}$ & $\lambda_{min}$ & $\sigma_{min}$ \\
 \hline
$3$	& $2.9940e-2$ & $1.2267e-2$ & $2.6333e-2$ & $1.2097e-2$ &	$7.1114e-3$ &	$5.2420e-3$ \\	
$4$	& $6.7992e-3$	& $5.4745e-3$ & $6.3025e-3$ & $5.3558e-3$ &	$5.8627e-3$ & $5.2003e-3$ \\	
$5$ & $7.1826e-3$ & $6.6451e-3$ & $3.0020e-3$ &	$2.9674e-3$ & $4.0691e-4$ & $3.5263e-4$ \\	
$6$	& $2.1129e-3$ & $2.0654e-3$ & $6.9654e-4$ &	$5.8389e-4$ & $4.1580e-4$	 & $3.2558e-4$ \\	
$7$	& $1.0044e-3$	& $4.2778e-4$ & $2.7894e-4$ & $2.2178e-4$ &	$2.1500e-5$ &	$1.6099e-5$ \\	
$8$ & $3.3227e-4$ & $3.2780e-4$ & $4.2257e-5$ & $1.8605e-5$ &	$2.4263e-6$ &	$1.0410e-6$ \\	
\hline
\end{tabular}
\caption{The minimal eigenvalue and singular value of $M^n$, $B_1^n$ and $B_2^n$}\label{tab:min}
\end{table}

We have also computed $\kappa_{\infty}(M^n)$, $\kappa_{\infty}(B_1^n)$ and $\kappa_{\infty}(B_2^n)$ for $n=3,\ldots,8$
with Mathematica. The results can be seen in Table \ref{tab:k}. It can be observed that 
$\kappa_\infty(M^n)\leq \kappa_\infty(B_i^n)$ for $i=1,2$, as it has been shown in Corollary \ref{cor:eig}.

\begin{table}[h!]
\centering\begin{tabular}{|r|c|c|c|}
\hline
$n$ & $\kappa_\infty(M^n)$ & $\kappa_\infty(B_1^n)$ & $\kappa_\infty(B_2^n)$ \\
 \hline
$3$ & $1.4138e+2$ & $1.4138e+2$ & $2.9393e+2$ \\
$4$ & $3.4704e+2$ & $3.4704e+2$ & $3.4704e+2$ \\
$5$ & $1.6822e+2$ & $4.3900e+2$ & $4.3526e+3$ \\
$6$ & $7.5191e+2$ & $3.1923e+3$ & $4.2045e+3$ \\
$7$ & $4.7287e+3$ & $5.5742e+3$ & $8.1522e+4$ \\ 
$8$ & $4.2039e+3$ & $1.2637e+5$ & $1.6388e+6$ \\
\hline
\end{tabular}
\caption{Infinity conditions numbers of $M^n$, $B_1^n$ and $B_2^n$}\label{tab:k}
\end{table}

\begin{rem}
On the one hand, we have seen in Corollary \ref{cor:eig} that the minimal eigenvalue and the minimal singular value of the
collocation matrix of the normalized B-basis are always greater than the minimal eigenvalue and the minimal
singular value, respectively, of the corresponding collocation matrix of the NTP bases of the
corresponding space of functions. This fact has also been illustrated in the previous numerical experiments. 
On the other hand, the maximal eigenvalue of the collocation matrix of an NTP basis of a space of functions, 
including the corresponding normalized B-basis, is always equal to $1$ because all these collocation matrices
are stochastic. So, an interesting
question arises: does there exist any relation between the maximal singular value of
the collocation matrices of the normalized B-basis of a space of functions and 
those of the corresponding collocation matrices of NTP bases of the same space?
In order to answer this question 
Table \ref{tab:max} also shows the maximal singular value of $M^n$, $B_1^n$ and $B_2^n$
for $n=3,\ldots,8$.
We can observe that in some cases the maximal singular value of $M^n$ is lower than
the maximal singular value of $B_1^n$ and $B_2^n$, for example for $n=5$. In other cases,
the maximal singular value of $M^n$ is higher than
the maximal singular value of $B_1^n$ and $B_2^n$, for example for $n=4$. Hence, we can conclude that there is not
a relation between the maximal singular value of the collocation matrix 
of a normalized B-basis and that of the corresponding collocation matrix of the NTP bases of the
corresponding space of functions. 

By Corollary \ref{cor:eig}, we have that $\kappa_{\infty}(M^n)\leq \kappa_{\infty}(B_i^n)$
and that $\sigma_{min}(M^n)\ge\sigma_{min}(B_i^n)$ for $i=1,2$ and $n=3,\ldots,8$. 
Taking into account that $\kappa_2(A)$ is equal to $\sigma_{max}(A)/\sigma_{min}(A)$, another interesting question
arises: does there exist an analogous relation with
$\kappa_2$ instead of $\kappa_\infty$ for the collocation matrices
of normalized B-bases and NTP bases? From the data in Tables \ref{tab:min}
and \ref{tab:max}, we have that $\kappa_2(M^3)>\kappa_2(B_1^3)$ and
$\kappa_2(M^5)<\kappa_2(B_i^5)$ for $i=1,2$. Hence, there is no any relation
between the condition number $\kappa_2$ of the collocation matrices 
of the normalized B-basis and these of the corresponding collocation matrices of NTP bases.
\end{rem}

\begin{table}[h!]
\centering\begin{tabular}{|r|c|c|c|}
\hline
$n$ & $\sigma_{max}(M^n)$ & $\sigma_{max}(B_1^n)$ & $\sigma_{max}(B_2^n)$ \\
 \hline
$3$ & $1.1934e+0$ & $1.1215e+0$ & $1.4619e+0$ \\
$4$ & $1.1074e+0$ & $1.0977e+0$ & $1.0542e+0$ \\
$5$ & $1.0608e+0$ & $1.0764e+0$ & $1.5601e+0$ \\
$6$ & $1.0709e+0$ & $1.1728e+0$ & $1.2136e+0$ \\
$7$ & $1.1237e+0$ & $1.4003e+0$ & $1.5872e+0$ \\
$8$ & $1.0968e+0$ & $1.3374e+0$ & $1.6461e+0$ \\
\hline
\end{tabular}
\caption{The maximal singular value of $M^n$, $B_1^n$ and $B_2^n$}\label{tab:max}
\end{table}


\end{document}